\documentclass{amsart}
\usepackage{amssymb,latexsym}
\usepackage{amsmath}
\usepackage{graphicx}

\newtheorem{thm}{Theorem}
\newtheorem{cor}[thm]{Corollary}

\newcommand{\Tn}{T_n}

\begin{document}
\title{The $N$-queens Problem on a symmetric Toeplitz matrix }

\author{Zsuzsanna Szaniszlo}
\address{Department of Mathematics, Valparaiso University } 
\email{zsuzsanna.szaniszlo@valpo.edu}

\author{Maggy Tomova}
\address{Department of Mathematics, Rice University}
\email{mt2@rice.edu}

\author{Cindy Wyels}
\address{Department of Mathematics, CSU Channel Islands} 
\email{cindy.wyels@csuci.edu}

\begin{abstract}
We consider the problem of placing $n$ nonattacking queens on a symmetric $n \times n$ Toeplitz matrix. As in the $N$-queens Problem on a chessboard, two queens may attack each other if they share a row or a column in the matrix.  However, the usual diagonal restriction is replaced by specifying that queens may attack other queens that occupy squares with the same number value in the matrix. We will show that $n$ nonattacking queens can be placed on such a matrix if and only if $n\equiv 0,1 \mod 4$.

\vspace{.1in}
\noindent \textbf{2000 AMS Subject Classification:} 05B99
\vspace{0.1in}

\noindent
{\bf Keywords}:  $N$-queens Problem
\end{abstract}

\maketitle

The $N$-queens Problem, a generalization of the original 8-queens problem, asks whether $n$ nonattacking queens can be placed on an $n \times n$ chessboard in such a way that no queen can attack another, i.e., so that no two queens are placed in the same row or column or on the same diagonal.  The problem has been extensively studied since the mid-1800s;  for a brief summary of the history, see \cite{HHS,RVZ,Wat}.  A familiar extension of the problem (due to Polya) asks the same question 
for queens placed on a toroidal chessboard. Vardi also considers the toroidal semiqueens problem, in which a semiqueen may move ``...like a rook or bishop, but only on positive broken diagonals." \cite{Vardi} (See Figure \ref{fig:Vardifig}.) To clarify the phrase ``positive broken diagonals" without reference to a figure, number the rows and columns of an $n \times n$ chessboard from 1 to $n$, starting with the top left square. The $k$th positive broken diagonal then consists of all squares labeled $(i,j)$ for $i+j=k+1 \mod n$.

\begin{figure}
\begin{center} \includegraphics[scale=.4]{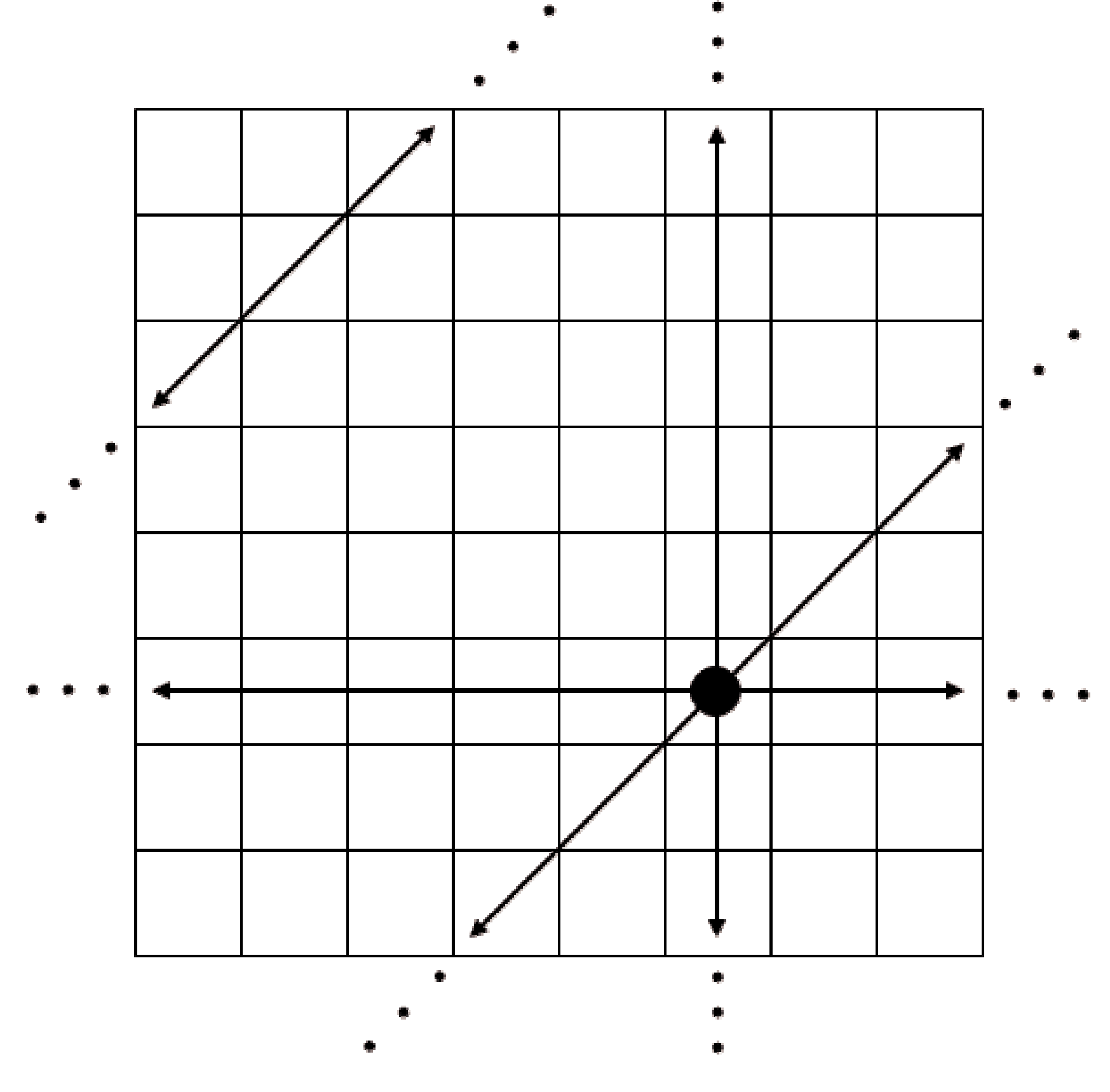}
\end{center} \caption{}
\label{fig:Vardifig} 
\end{figure}

In this paper we modify the chessboard again; now it corresponds to a symmetric Toeplitz matrix. A Toeplitz matrix has constant negative diagonals, i.e., entry $(i_1,j_1)$ equals entry $(i_2,j_2)$ whenever $i_1-j_1 =i_2-j_2$ \cite{PFTV}. In a symmetric Toeplitz matrix, we further require that entry $(i_1,j_1)$ equal entry $(i_2,j_2)$ whenever $|i_1-j_1| =|i_2-j_2|$. Queens may attack each other if they share a row or a column, or if both are on squares belonging to a set of the form $D_k = \{(i,j)\, |\, |i-j| = k\}$. We then ask when $n$ nonattacking queens can be placed on such an $n \times n$ chessboard. We will show that $n$ nonattacking queens can be placed on such a chessboard if and only if $n\equiv 0,1\mod 4$.

\begin{figure}
\begin{center} \includegraphics[scale=.4]{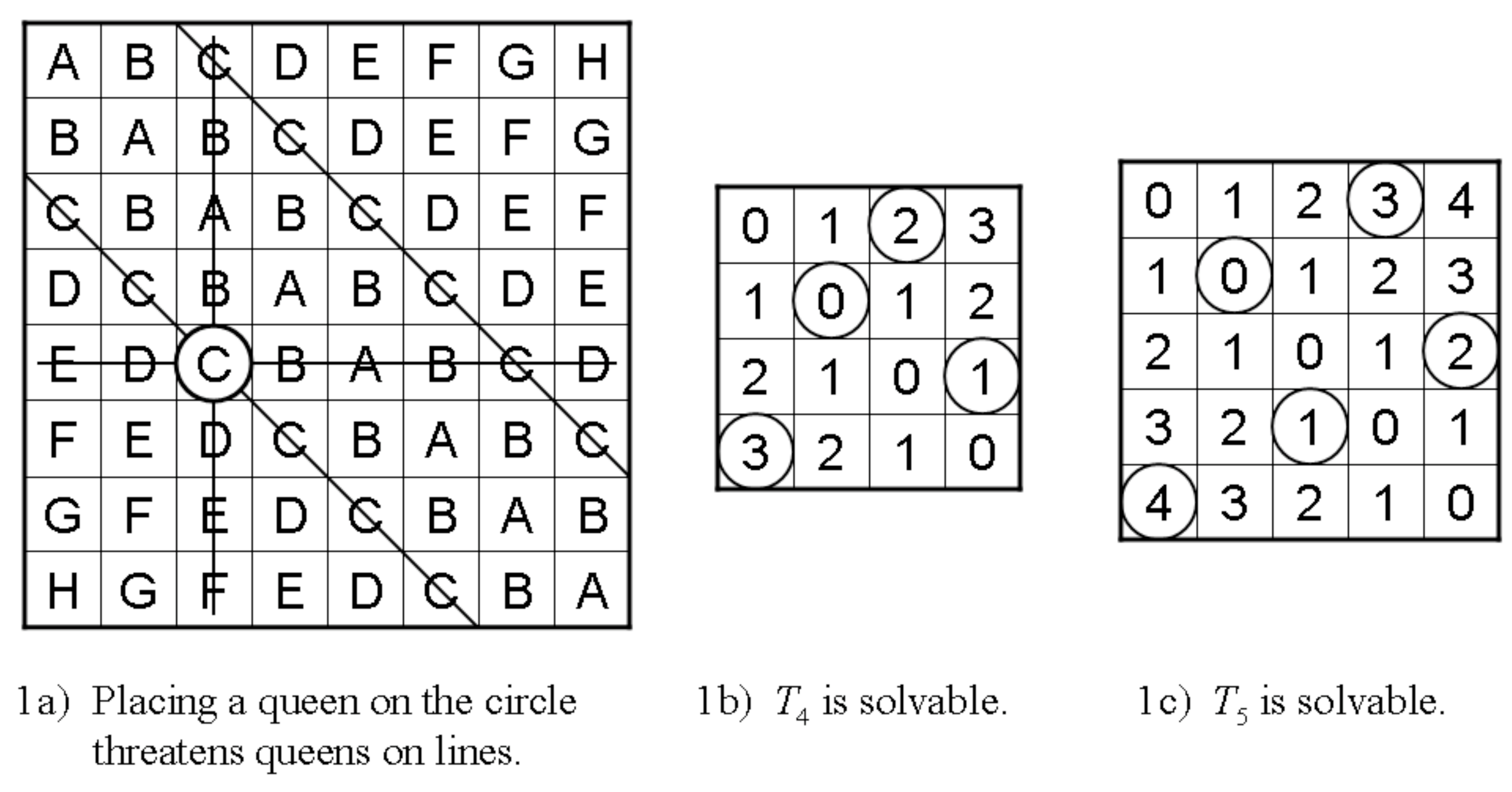}
\end{center}\caption{}
\label{fig:toep}
\end{figure}

To solve this problem we will consider the matrix $T_n$ with entries
given by $t_{ij}=|i-j|$, for $i,j = 1, \ldots ,n$.  We say that $T_n$ is solvable if and only if we can select $n$ entries from $T_n$ with values $0,\ldots ,n-1$ so that no two entries selected lie in the same row or the same column. 
Equivalently, $T_n$ is solvable if and only if there
exists a permutation  $ f:\{1,2,\dots ,n\} \to \{1,2,\dots ,n\}$
such that $|f(i)-i|$ assumes each of the values $\{0,1,\dots ,n-1\}$
exactly once.  (We will use both formulations throughout this paper.)  A solution set for $T_n$ is a set of pairs $S_n=\{(i,f(i))\, |\, i=1,\ldots ,n\}$ where $f$ is a permutation demonstrating $T_n$'s solvability.  For example, as indicated in Figure \ref{fig:toep}, solution sets for $T_4$ and $T_5$ are, respectively, $S_4=\{(1,3),(2,2),(3,4),(4,1)\}$ and $S_5 = \{(1,4),(2,2),(3,5),(4,3),(5,1)\}$.

\begin{thm} \label{thm:solvable}
$T_n$ is solvable if and only if $n\equiv 0,1 \mod 4$.
 \end{thm}

\begin{proof}
We use a counting argument to show that $T_n$ is not solvable for $n \equiv 2,3 \mod 4$, then construct a solution set for $T_n$ when $n \equiv 0, 1 \mod 4$.

Suppose $f:\{1,2,\dots ,n\} \to \{1,2,\dots ,n\}$ is a solution for
$T_n$. As $f$ is a permutation, it follows that

\begin{equation}  \label{E:first}
\sum _{i=1} ^n \bigl ( f(i) \bigr ) ^2 = \sum _{i=1} ^n i^2 =
\frac{n(n+1)(2n+1)}6.
\end{equation}
As $|f(i)-i|$ achieves each of the values $\{0,\ldots ,n-1\}$ exactly
once, we have

\begin{equation}  \label{E:second}
\sum _{i=1} ^n |f(i)-i|^2 = \sum _{j=0} ^{n-1} j^2 =
\frac{(n-1)n(2n-1)}6.
\end{equation}
However, it is also the case that
\[
\sum _{i=1} ^n |f(i)-i|^2 = \sum _{i=1} ^n \bigl ( f(i)-i \bigr )^2
= \sum _{i=1} ^n \bigl ( f(i) \bigr ) ^2 -2\sum _{i=1}^n if(i) +
\sum _{i=1} ^n i^2.
\]
Simplifying, we obtain
\[
\sum _{i=1} ^n |f(i)-i|^2 =2\sum _{i=1} ^n i^2 - 2\sum _{i=1}^n
if(i).
\]
Using \eqref{E:first} and \eqref{E:second} gives
\[
\frac{(n-1)n(2n-1)}6 = 2\bigl( \frac{n(n+1)(2n+1)}6\bigr) - 2\sum
_{i=1}^n if(i)
\]
or
\[
\sum _{i=1}^n if(i) = \frac 1{12} \bigl ( 2n(n+1)(2n+1)-(n-1)n(2n-1)
\bigr ).
\]
Simplifying the right-hand side and recognizing that the left-hand
side is an integer yields
\[
n(2n^2+9n+1) \equiv 0 \mod 12.
\]
Suppose first that $n \equiv 2 \mod 4$. Since $n$ is even, $2n^2+9n+1$ is odd. But this together with $ n(2n^2+9n+1) \equiv 0 \mod 12$ implies $n \equiv 0 \mod 4$, a contradiction. Thus no solutions exist for $n \equiv 2 \mod 4$. Now assume that $n\equiv 3 \mod 4$. A short calculation shows that $2n^2+9n+1 \equiv 2 \mod 4$, so $ n(2n^2+9n+1) \equiv 2n  \equiv 2 \mod 4$ (under the assumption
that $n \equiv 3 \mod 4$). But $n(2n^2+9n+1) \equiv 2 \mod 4$ contradicts $ n(2n^2+9n+1) \equiv 0 \mod 12$. Therefore there are no solutions when $n \equiv 3 \mod 4$.

\medskip

It remains to show that $T_n$ is solvable when $n\equiv 0,1$ mod $4$. The proof will be done by induction on $n$. 

We begin by defining two matrices that can be obtained from $T_n$. 
\begin{itemize}
\item $T_n^*$ is the matrix obtained from $T_n$ by deleting the first column and the last row of $T_n$.
\item $T_n^{**}$ is obtained from  $T_n$ by deleting the first and last row as well as the first and the $(n-1)^{st}$ column of $T_n$.
\end{itemize}
Figure \ref{fig:derive} illustrates $T_n^*$ and $T_n^{**}$. We will call $T_n^*$ solvable if we can select entries with values $0,\ldots,n-2$, one from each row and each column;  likewise $T_n^{**}$ will be called solvable if we can select entries with values $0,\ldots,n-3$, one from each row and each column. Note that if $T_n^*$ is solvable, then so is its transpose, $(T_n^*)^T$, which is obtained from $T_n$ by removing the first row and the last column. We will not distinguish between $T_n^*$ and $(T_n^*)^T$. The same holds for $T_n^{**}$ and $(T_n^{**})^T$.

\begin{figure}
\begin{center} \includegraphics[scale=.7]{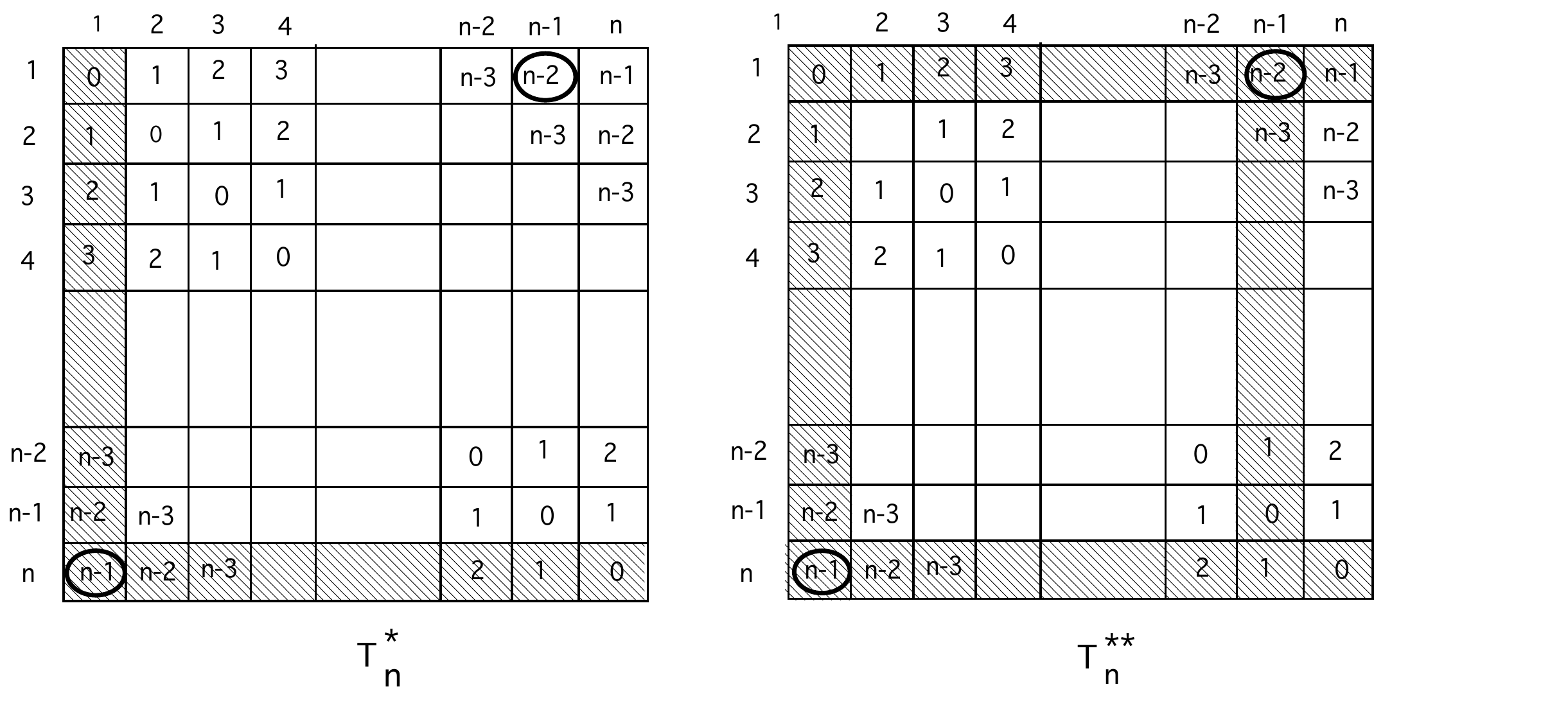}
\end{center} \caption{}
\label{fig:derive}
\end{figure}

If $\Tn$ is solvable, without loss of generality we can assume that the $(n,1)$ entry was selected as the entry of value $n-1$, see Figure \ref{fig:derive}. By symmetry we may assume that the $(1, n-1)$ entry was selected as the entry of value $n-2$. It is easy to see that removing the $(n,1)$ entry from the solution set of $\Tn$ gives a solution set for $T_n^*$ and removing both the $(n,1)$ and the $(1, n-1)$ entries gives a solution for $T_n^{**}$. Similarly adding one or both of these entries to the solution of $T_n^{**}$ results in solutions for $T_n^*$ and $\Tn$ respectively. We conclude that, for fixed $n$, either all three of $T_n, T_n^{*}$ and $T_n^{**}$ are solvable or none of them is.

Solutions for $T_4$ and $T_5$ are provided in Figure \ref{fig:toep}.  We assume the result for all $k<n$. The solution set depends on the remainder of $n \mod 3$, so we write $n=3r+s$, where $s$ may be 0, 1, or 2. Cases 2 and 3 are similar to Case 1 and thus will be written out in less detail.

\textbf{Case 1: $n=3r+0$}. We build a solution set $S_n$ as follows: Select entries $(n,1),(n-1,2), \ldots,(n-r+1,r)=(2r+1,r)$ which have values $n-1, n-3,\ldots,n-2r+1=r+1$ respectively. Also select entries $(1,n-1),(2,n-2),\ldots,(r-1,n-(r-1))=(r-1,2r+1)$ which have values $n-2, n-4,\ldots,n-2r+2=r+2$ and finally select the entry $(n-r,n)=(2r,n)$ with value $r$. Thus far $S_n$ contains all values $r,r+1,\ldots,n-1$ and no two selections are in the same row or column.

The entries that are still available to be selected (i.e., those lying in a row or column not yet selected) have indices $(a,b)$ with $r \leq a \leq 2r-1$ and $r+1\leq b \leq 2r$.  As each entry is equal to the difference between its row and column index, we can subtract $r-1$ from each of the row and column indices and obtain the matrix with row labels $1$ through $r$ and column labels $2$ through $r+1$. The matrix we obtain is $T_{r+1}$ with the last row and the first column deleted, i.e., it is $T_{r+1}^*$. Thus if $T_{r+1}^*$ has a solution set $S_{r+1}^*$, adding this solution set to the entries we already selected completes $S_n$. By induction we know $T_{r+1}^*$ is solvable as long as $r+1\equiv 0 \text{ or } 1 \mod 4$. However it is easy to check that $n\equiv 0 \text{ or } 1 \mod  4$ if and only if $r+1\equiv 0 \text{ or }  1 \mod  4$.

\textbf{Case 2: $n=3r+1$.} Select entries $(n,1),(n-1,2), \ldots,(n-r+1,r)=(2r+2, r)$ and $(1,n-1),(2,n-2),\ldots,(r,n-r)=(r, 2r+1)$ and $(n-r,n)=(2r+1,n)$. So far we have selected values with entries $r,r+1,\ldots,n-1$.  The entries that are still available to be selected have indices $(a,b)$ such that $r+1 \leq a \leq 2r$ and $r+1\leq b \leq 2r$. By subtracting $r$ from both the column and row labels, we see that the matrix that remains is $T_r$. Again, it is easy to check that $n\equiv 0 \text{ or } 1 \mod  4$ if and only if $r\equiv 0$ or 1 $\mod  4$.  Thus, by the induction hypothesis, we can complete the solution set $S_n$.

\textbf{Case 3: $n=3r+2$.} Begin by selecting entries $(n,1),(n-1,2), \ldots,(n-r+1,r)=(2r+3,r)$ and $(1,n-1),(2,n-2),\ldots,(r,n-r)=(r, 2r+2)$ and $(n-(r+1),n)=(2r+1,n)$.  The values we have selected so far are $r+1,\ldots,n-1$. The entries that are still available to be selected are those with indices $(a,b)$ such that $r+1 \leq a \leq 2r$ or $a=2r+2$ and $r+1 \leq b \leq 2r+1$.  By subtracting $r-1$ from all row and column labels we see that these are the entries of the transpose of $T_{r+3}^{**}$ which by the induction hypothesis is solvable as long as $r+3 \equiv 0 \text{ or } 1 \mod 4$.  Again we see that $n \equiv 0 \text{ or } 1 \mod 4$ if and only if $r+3 \equiv 0 \text{ or } 1 \mod 4$, and thus we can complete the solution set $S_n$.
\end{proof}

A symmetric Toeplitz chessboard is an $n \times n$ chessboard in which the squares are labeled with numbers so that the labels along each descending diagonal are constant and the labeling is symmetric with respect to the main descending diagonal. A queen placed on the chessboard may attack other queens in the same row or column as well as those placed on a square with the same number value.

\begin{cor}\label{thm:queencor}
It is possible to place $n$ queens on an $n \times n$ symmetric Toeplitz chessboard so that no two queens may attack each other if and only if $n\equiv 0,1 \mod 4$.

\end{cor}
 \begin{proof}
Without loss of generality assume the squares of the chessboard are labeled with the entries of $T_n$. Then placing the nonattacking queens on the chessboard is equivalent to finding a solution for $T_n$. The result follows by Theorem \ref{thm:solvable}.

 \end{proof}

Corollary \ref{thm:queencor} establishes the existence (and nonexistence) of solutions to the $N$-queens Problem on the particular modification of the chessboard discussed in this paper. However, variants of this problem analogous to those on other chessboards remain to be addressed. For instance:
\begin{enumerate}
\item What is the minimum number of queens necessary for each square of the board to either contain a queen or to be attacked by at least one queen? (What is the minimum cardinality of a dominating set, i.e., what is the domination number?)
\item When $n\equiv 2, 3 \mod 4$, what is the maximum number of queens that can be placed on the board so that no queen can attack another? (What is the maximum cardinality of an independent dominating set?)
\item Substitute another type of chess piece for queens and ask analogous questions.
\item How many (fundamental) solutions exist for each question?
\end{enumerate}

Theorem \ref{thm:n-1} resolves the second question.
\begin{thm} \label{thm:n-1}
It is possible to place $n-1$ nonattacking queens on an $n \times n$ chessboard corresponding to a symmetric Toeplitz matrix.
\end{thm}
\begin{proof}
We provide such a placement of $n-1$ nonattacking queens as follows. 
\begin{enumerate}
\item Place a queen on square $(1,1)$. 
\item For each $i$ from 1 to $\lceil \frac n2 \rceil -1$, place a queen on square $(n+1-i, i+1)$. 
\item For each $j$ from 3 to $\lfloor \frac n2 \rfloor +1$, place a queen on square $(j, n+3-j)$. 
\end{enumerate}
Figure \ref{fig:n-1fig} indicates this placement. To facilitate detail-checking, we examine $n$ even and $n$ odd separately.

\begin{figure}
\begin{center} \includegraphics[scale=.4]{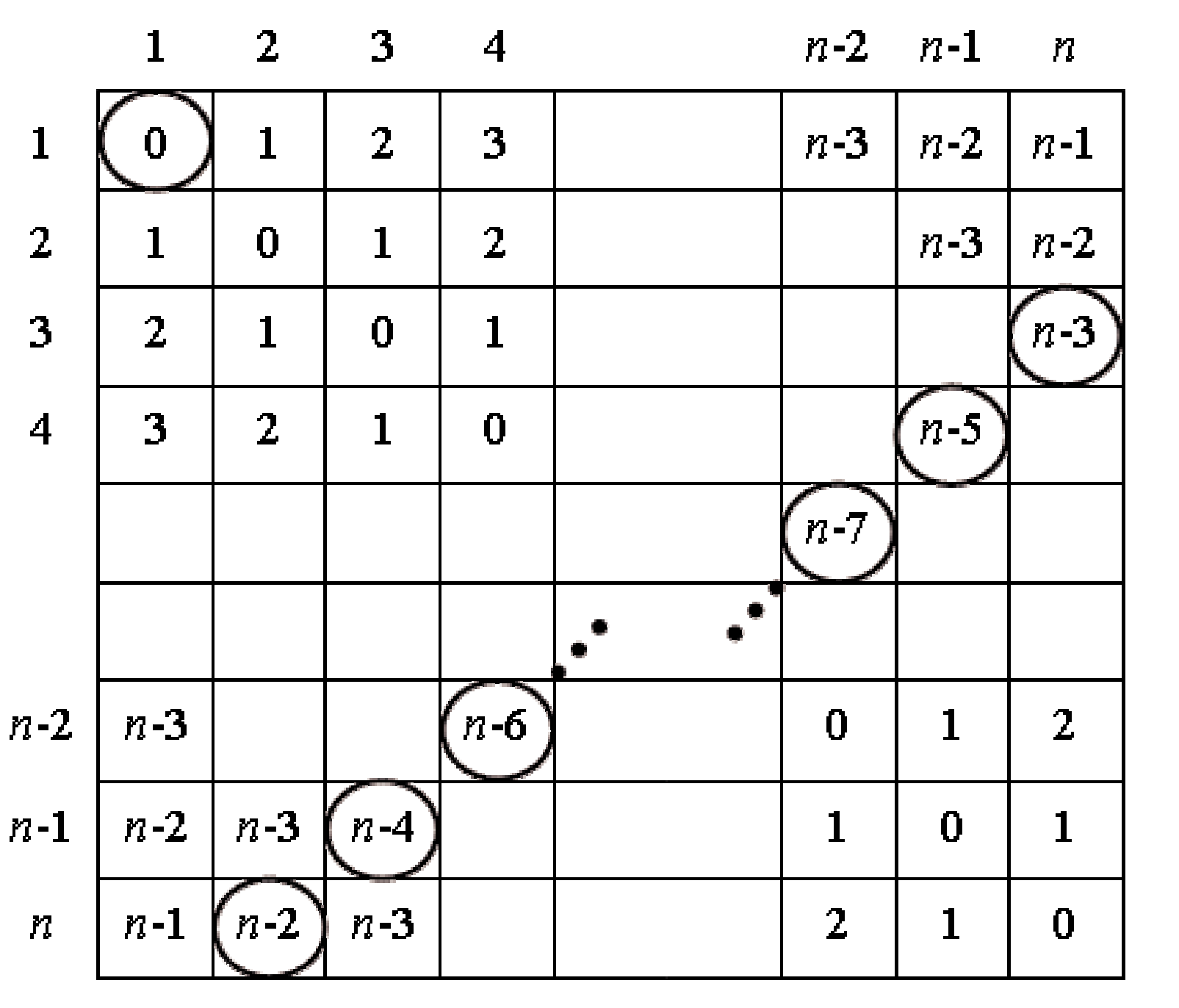}
\end{center} \caption{}
\label{fig:n-1fig}
\end{figure}

\textbf{Case 1: $n$ is even}.
Examine the rows occupied by queens. Step 1 places a queen in row 1; step 2 places queens in rows $ \frac n2 +2, \dots n$; step 3 places queens in rows $3, \dots , \frac n2 +1$. The $n-1$ queens occur once in every row with the exception of row 2 only. Similarly, we have queens in column 1 (step 1), in columns $2, \dots  \frac n2$ (step 2), and in columns $\frac n2 +2,\dots, n$ (step 3). Every column except column $\frac n2 +1$ has exactly one queen. Calculating the differences of the row and column indices shows that there is exactly one queen in each of the diagonals indexed by 0 (step 1), by $2, 4, \dots , n-2$ (step 2), and by $1, 3, \dots n-3$ (step 3). (The diagonal indexed by $n-1$ has no queen.)

\textbf{Case 2: $n$ is odd}. Rows occupied by a queen are row 1 (step 1), rows $\frac{n+3}2, \dots n$ (step 2), and rows $3, \dots \frac{n+1}2$ (step 3). Only row 2 has no queen, so the other rows have one queen each. Columns occupied by a queen are column 1 (step 1), columns $2, \dots \frac{n+1}2$ (step 2), and columns $\frac{n+5}2, \dots n$ (step 3). The only column with no queen is column $\frac {n+3}2$, so again we conclude that all the remaining columns contain exactly one queen. (Note that, as $n$ is odd, $\frac {n+3}2$ may be written as $\lceil \frac n2\rceil+1$, and $\lceil \frac n2\rceil+1 = \frac n2+1$ when $n$ is even, so the expression $\lceil \frac n2\rceil+1$ identifies the column without a queen for all values of $n$.) Finally, we check the diagonals occupied by queens: diagonal 0 (step 1), diagonals $1, 3, \dots , n-2$ (step 2), and diagonals $2, 4, \dots , n-3$ (step 3). As before, only the diagonal indexed by $n-1$ has no queen; the others have exactly one each. 
\end{proof}

 \section*{Aknowledgements}

The authors would like to thank the Department of Mathematics and Computer Science at Valparaiso University for hosting us while part of this research was carried out, as well as two referees whose helpful comments improved the final version.  The third author also thanks CSU Channel Islands for the Mini-Grant which provided time to work on this project.

\end{document}